\documentclass[11pt,a4paper,oneside,english,reqno]{amsart}
\usepackage[T1]{fontenc}
\usepackage[latin9]{inputenc}
\usepackage[left=2.5cm,right=2.5cm,top=2.3cm,bottom=2.5cm]{geometry}
\usepackage{amsthm}
\usepackage{float}
\usepackage{amssymb}

\newtheorem{theorem}{Theorem}[section]

\usepackage{color}
\usepackage{pgf,tikz}
\usepackage{graphicx}%
\newtheorem{remark}{Remark}[section]

\numberwithin{equation}{section}
\newcommand{\ds}{\displaystyle}
\numberwithin{equation}{section}
\numberwithin{figure}{section}
\newcommand{\W}{H_0^{1}(\Omega)}
\newcommand{\R}{{\mathbb R}}
\usepackage{color}

\usepackage{hyperref}
\hypersetup{
    colorlinks = true,%
    citecolor = [rgb]{0.0,0.0,0.9},
    filecolor=black,%
    linkcolor = [rgb]{0.65,0.0,0.0},%
    anchorcolor = red,
    pagecolor = red,
    urlcolor= [rgb]{0.65,0.0,0.0}
}

\usepackage{mathrsfs}
\makeatother

\usepackage{babel}
\title{Energy properties of critical Kirchhoff problems with applications}

\author{Francesca Faraci }

\address{Department of Mathematics and Computer Science, University of Catania,
	Catania, Italy}\email{ffaraci@dmi.unict.it}

\author{Csaba Farkas}
\address{Department of Mathematics and Computer Science, Sapientia University,
	Tg. Mures, Romania \& Institute of Applied Mathematics, \'Obuda University,
	1034 Budapest, Hungary}
\email{farkas.csaba2008@gmail.com \& farkascs@ms.sapientia.ro}

\author{Alexandru Krist\'aly}
\address{\noindent Department of Economics, Babe\c s-Bolyai University, Cluj-Napoca, Romania\ \&  Institute of Applied Mathematics,
	\'Obuda University, 1034 Budapest, Hungary}
\email{alex.kristaly@econ.ubbcluj.ro; kristaly.alexandru@nik.uni-obuda.hu}

\thanks{The research of C. Farkas and A. Krist\'aly is supported by the National Research, Development and Innovation Fund of Hungary, financed under the K$\_$18 funding scheme, Project No.  127926.}

\begin{document}

\maketitle
\begin{abstract}
	In this paper we fully characterize the sequentially weakly lower semicontinuity of the parameter-depending energy functional associated with the critical Kirchhoff problem. We also establish sufficient criteria with respect to the parameters for the convexity and validity of the Palais-Smale condition of the same energy functional. We then apply these regularity properties in the study of some elliptic problems involving the critical Kirchhoff term.
\end{abstract}
\section{Introduction}
The time-depending state of a stretched string is given by the solution of the nonlocal equation
\begin{equation}
\label{Kirchhoff}
u_{tt}- \left( a+b\ds\int_\Omega |\nabla_x u|^2 dx\right)\Delta_x u=h(t,x,u), \ \ (t,x)\in (0,+\infty)\times\Omega,
\end{equation}
proposed first by Kirchhoff \cite{Kirchhoff} in 1883. 
% where $\Omega\subset \mathbb R^d$ $(d\geq 2)$ is a bounded open set, $a,b>0$ are parameters and .  
%Which could be viewed as a generalization of the well-known D' Alembert wave equation.
In \eqref{Kirchhoff}, $\Omega\subset \mathbb R^d$ is an open bounded domain, the solution $u:(0,\infty)\times\Omega \to \mathbb{R}$ denotes the displacement of the string,  $h:(0,\infty)\times\Omega\times \mathbb R\to \mathbb R$ is a Carath\'eodory function  representing the
external force,  $a$ is the initial tension, while $b$ is related to the intrinsic properties of the string (such as Young's modulus of the material).
Other nonlocal equations similar to \eqref{Kirchhoff} appear also in  biological systems, where $u$ describes a process depending on its average (over a given set), like population density, see e.g. Chipot and Lovat \cite{CL}.

Let $\Omega\subset\mathbb{R}^{d}$ be an open bounded domain, $d\geq4.$ The primary aim of the present paper is to establish basic properties of the energy functional associated with the stationary form of \eqref{Kirchhoff}, involving a critical term and subject to the Dirichlet boundary  condition, namely, 
\[
\ \left\{ \begin{array}{lll}
{\displaystyle -\left(a+b\int_{\Omega}|\nabla u|^{2}\right)\Delta u=|u|^{2^{*}-2}u} & \mbox{in} & \Omega,\\
u=0 & \mbox{on} & \partial\Omega,
\end{array}\right.\eqno{(\mathscr{P}_{a,b})}
\]
where $a,b>0$ and $\ds 2^*=\frac{2d}{d-2}$ is the critical Sobolev exponent. 
In spite of the competing effect of the nonlocal term $\ds \int_{\Omega}|\nabla u|^2dx\Delta u$ with the critical nonlinearity $|u|^{2^*-2}u$ as well as the lack of compactness of the Sobolev embedding $H^{1}_0(\Omega)\hookrightarrow L^{2^\star}(\Omega)$,  there are several contributions concerning existence and multiplicity of solutions for problem $(\mathscr{P}_{a,b})$, by  using fine arguments both from variational analysis (see e.g. Autuori, Fiscella and Pucci \cite{AFP}, Chen, Kuo and Wu \cite{CKW}, Corr\^{e}a and Figueiredo \cite{CF}, Figueiredo \cite{F}, Perera and Zhang\cite{PZ,ZP}) and topological methods (see e.g. Fan \cite{Fan},  Figueiredo and Santos \cite{FS}).
It is also worth mentioning that the Palais-Smale compactness condition combined with  the Lions concentration compactness principle \cite{L} are still the most popular tools to deal with elliptic problems involving critical terms. We note that problem $(\mathscr{P}_{a,b})$ is sensitive with respect to the size of the space dimension $d$. Indeed, different arguments/results are applied/obtained for the lower dimensional case $d\in \{3,4\} $ (see e.g. Alves, Corr\^{e}a and Figueiredo \cite{ACF}, Deng and Shuai \cite{DS}, Lei,  Liu and Guo \cite{LLG} and Naimen \cite{N0}) and  for the higher dimensional case $d>4$ (see Alves,  Corr\^{e}a and Ma \cite{ACM}, Hebey \cite{H1,H2}, Yao and Mu \cite{YM}); moreover,  the parameters $a$ and $b$ should satisfy suitable constraints in order to employ the aforementioned principles.

In order to obtain qualitative results  in the theory of Kirchhoff problems via direct methods of the calculus of variations (see e.g. Dacorogna \cite{Dacorogna}), basic regularity properties of the energy functional $\mathcal{E}_{a,b}$  associated with  problem $(\mathscr{P}_{a,b})$ are needed. Accordingly, in terms of $a$ and $b$, we fully characterize the sequentially weakly lower semicontinuity of $\mathcal{E}_{a,b}$, see Theorem \ref{main1}/(i). In addition,  sufficient conditions are also provided for both the validity of the Palais-Smale compactness condition and convexity  of $\mathcal{E}_{a,b}$, see Theorem\ref{main1}/(ii) and (iii), respectively.

 In the sequel, we state  our main results. Let $H^1_0(\Omega)$ and $L^q(\Omega)$ ($1\leq q\leq 2^*$) be the usual Sobolev  and Lebesgue spaces  endowed with the norms $$\|u\|=\left(\int_\Omega |\nabla u|^2\right)^{\frac{1}{2}}\ \ \mbox{ and }\ \ \|u\|_q=\left(\int_\Omega | u|^q\right)^{\frac{1}{q}},$$ respectively. 
The  critical Sobolev inequality is given by  
\begin{equation}\label{defSobolev}\mathtt{S}_d=\inf_{u\in H_0^1(\Omega)\setminus\{0\}}\frac{\|u\|^2}{\|u\|_{2^*}^2},\end{equation} 
or
\begin{equation}\label{Sob}
\|u\|_{2^*}^{2}\leq \mathtt{S}_d^{-1}\|u\|^2,\ \ \forall u\in H_0^1(\Omega),
\end{equation}
%\textcolor{blue}{	Let $S_N$ be the embedding constant of $H^1_0(\Omega)\hookrightarrow L^{2^\star}(\Omega)$, i.e.
%	\[\|u\|^2_{2^\star}\leq S_N^{-1} \|u\|^2 \qquad \mbox{for every } \ u\in H^1_0(\Omega). \]
%	Let us recall that
%	\begin{equation}\label{2*}
%	S_N=\frac{N(N-2)}{4}\omega_N^{\frac{2}{N}},
%	\end{equation}
%	where $\omega_N$ is the volume of the unit ball in $\mathbb{R}^N$.
%	For $N>4$ denote by $C(N)$ and $C_2(N)$ the constants
%	\[
%	C(N)=\frac{4(N-4)^\frac{N-4}{2}}{N^\frac{N-2}{2}S_N^\frac{N}{2}},
%	\]
%	and
%	\[
%	C_2(N)=\frac{2(N-4)^{\frac{N-4}{2}}}{(N-2)^{\frac{N-2}{2}}S_{N}^{\frac{N}{2}}}.
%	\]
%	Notice that $C(N)< C_2(N)$.
%}
where
\begin{equation}\label{2*}
\mathtt{S}_d=\frac{d(d-2)}{4}\omega_d^{\frac{2}{d}},
\end{equation}
 see Talenti \cite{Talenti}, $\omega_d$ being the volume of the unit ball in $\mathbb{R}^d$. Note that the constant $\mathtt{S}_d$ is sharp in (\ref{defSobolev}) but never achieved except when $\Omega=\mathbb{R}^d$, see e.g. Willem \cite{Willem}.
The energy functional $\mathcal{E}_{a,b}:H^1_0(\Omega)\to\mathbb{R}$ associated with  problem $(\mathscr{P}_{a,b})$ is defined by $$\mathcal{E}_{a,b}(u)=\frac{a}{2}\|u\|^2+\frac{b}{4}\|u\|^4-\frac{1}{2^*}\|u\|_{2^*}^{2^*}.$$

For a fixed $d\geq4,$ we introduce the constants
\[
\mathtt{L}_d=
\begin{cases}\ds
\frac{4(d-4)^\frac{d-4}{2}}{d^\frac{d-2}{2}\mathtt{S}_d^\frac{d}{2}}, & d>4\\ \\
\ds \frac{1}{\mathtt{S}_{4}^{2}}, & d=4,
\end{cases}
\qquad  \qquad
\mathtt{PS}_d=\begin{cases}
\ds\frac{2(d-4)^{\frac{d-4}{2}}}{(d-2)^{\frac{d-2}{2}}\mathtt{S}_{d}^{\frac{d}{2}}}, & d>4\\ \\
\ds \frac{1}{\mathtt{S}_{4}^{2}}, & d=4,
\end{cases}
\]
and
$$\mathtt{C}_d=
\begin{cases}\ds
\frac{2(d-4)^{\frac{d-4}{2}}(d+2)^{\frac{d-2}{2}}}{(d-2)^{d-2}\mathtt{S}_d^{\frac{d}{2}}}, & d>4\\ \\
\ds \frac{3}{\mathtt{S}_{4}^{2}}, & d=4,
\end{cases}$$
which will play crucial roles in the lower semicontinuity, validity of the PS-condition and convexity of $\mathcal{E}_{a,b}$, respectively. Note that for every $d\geq 4$, we have
\begin{equation}\label{sorrend}
\mathtt{L}_d\leq\mathtt{PS}_d\leq \mathtt{C}_d.
\end{equation}
Moreover, as a formal observation, we notice that 
$$\lim_{d\to 4}\mathtt{L}_d=\mathtt{L}_4;\ \lim_{d\to 4}\mathtt{PS}_d=\mathtt{PS}_4;\ \lim_{d\to 4}\mathtt{C}_d=\mathtt{C}_4.$$

Our main result reads as follows:
\begin{theorem}\label{main1} Let $\Omega\subset \mathbb{R}^d$ be an open bounded domain $(d\geq 4)$,  $a,b>0$ two fixed numbers, and $\mathcal{E}_{a,b}$ be the  energy functional  associated with problem $(\mathscr{P}_{a,b})$. Then the following statements hold: 
	\begin{itemize}
		\item[{\rm (i)}]  $\mathcal{E}_{a,b}$ is sequentially weakly lower semicontinuous on $H^1_0(\Omega)$ if and only if  $a^{\frac{d-4}{2}}b\geq \ds\mathtt{L}_d;$
			\item[{\rm (ii)}] $\mathcal E_{a,b}$  satisfies the Palais-Smale condition on $H^1_0(\Omega)$ whenever $\ds a^{\frac{d-4}{2}}b>\mathtt{PS}_{d};$
				\item[{\rm (iii)}] $\mathcal E_{a,b}$ is convex on $H^1_0(\Omega)$ whenever $\ds a^{\frac{d-4}{2}}b\geq \mathtt{C}_{d}.$ In addition, $\mathcal{E}_{a,b}$ is strictly convex on $H^1_0(\Omega)$ whenever $\ds a^{\frac{d-4}{2}}b>\mathtt{C}_{d}$. 
	\end{itemize}
\end{theorem}

\begin{remark}\rm 
	(i) By the proof of Theorem \ref{main1}/(i) we observe that the sequentially weakly lower semicontinuouity of $\mathcal{E}_{a,b}$ holds on \textit{any} open domain $\Omega\subseteq \mathbb{R}^d$ (not necessary bounded). However, the optimality of the constant $\mathtt{L}_d$ requires that $\Omega\neq \mathbb{R}^d$, see Section \ref{proofs}.
	
	(ii) Note that a similar result as Theorem \ref{main1}/(ii)  (with the same assumption $\ds a^{\frac{d-4}{2}}b>\mathtt{PS}_{d}$) has been proved by Hebey \cite{H2} on compact Riemannian manifolds. We provide here a genuinely different proof than in \cite{H2} based on the second concentration compactness lemma of Lions \cite{L}. 
	
%	(iii) The regularity results from Theorem \ref{main1} can be extended to Riemannian manifolds with suitable modifications. The characterization 
\end{remark}

\begin{figure}
	\centering
	\includegraphics[scale=0.55]{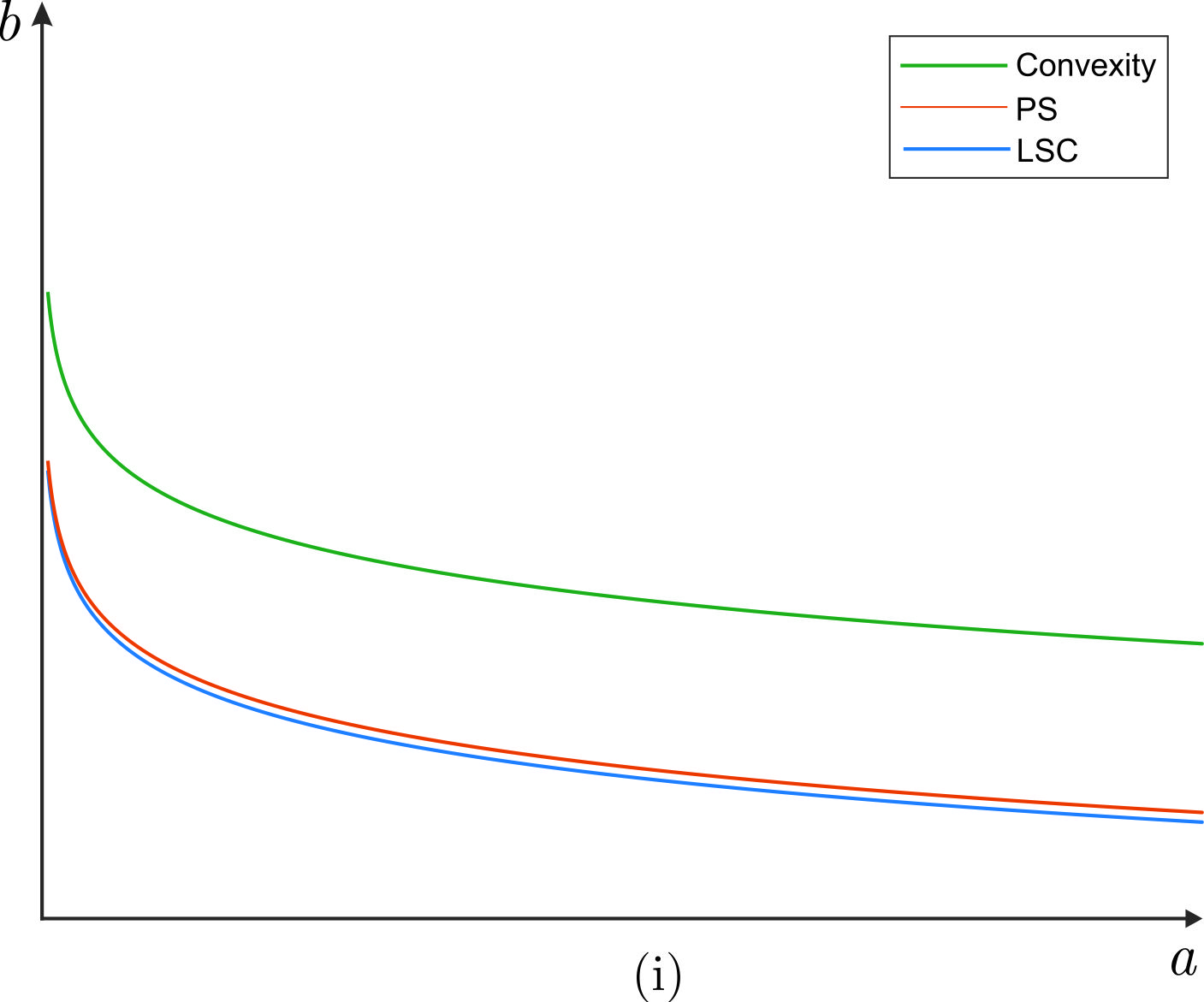} \ \
	\includegraphics[scale=0.55]{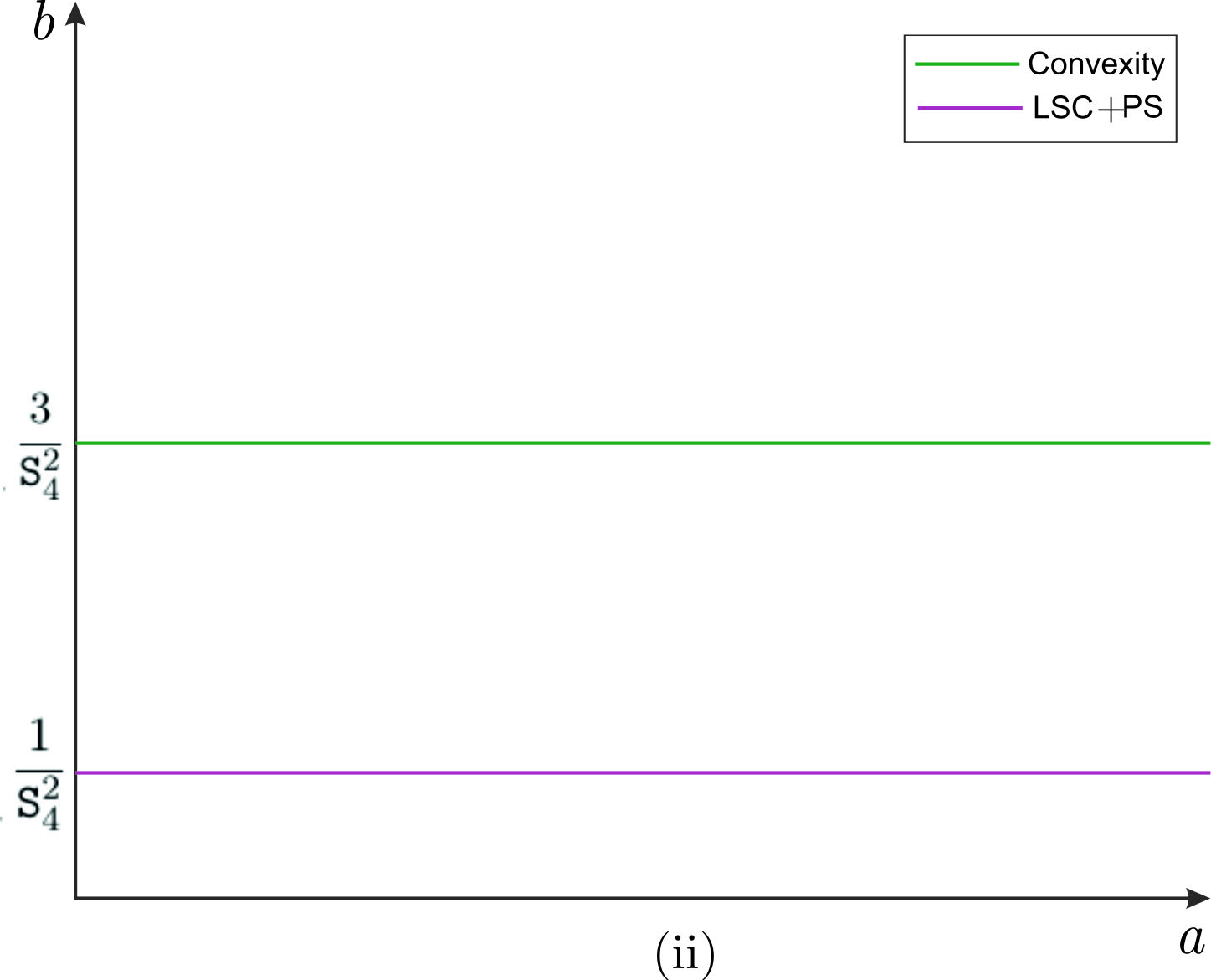}

	\caption{Curves $a^{\frac{d-4}{2}}b=\mathtt{L}_{d} \mbox{ and }\mathtt{PS}_d\mbox{ and }\mathtt{C}_d$ for $d>4$ (case (i)) and  $d=4$ (case (ii)).}
\end{figure}

In the sequel, we provide two applications of Theorem \ref{main1}. First, we  consider the model Poisson type problem
\[
\ \left\{ \begin{array}{lll}
{\displaystyle -\left(a+b\int_{\Omega}|\nabla u|^{2}\right)\Delta u=|u|^{2^{*}-2}u+h(x)} & \mbox{in} & \Omega,\\
u=0 & \mbox{on} & \partial\Omega,
\end{array}\right.\eqno{(\mathscr{P}_{a,b}^h)}
\]
where $h\in L^\infty(\Omega)$ is a positive function.

%
%As application of Theorem \ref{main1}, we consider perturbations of problem $(\mathscr{P}_{a,b})$,  which can be written as 
%$$
%\left\{
%\begin{array}{ll}
%{\displaystyle -\left(a+b\int_{\Omega}|\nabla u|^{2}\right)\Delta u=|u|^{2^{*}-2}u}+\lambda f(x,u), & \hbox{ in } \Omega \\
%
%u=0, & \hbox{on } \partial \Omega
%\end{array}
%\right. \eqno{(\mathscr{P}_{a,b}^f)}
%$$
%where $\Omega\subset \mathbb{R}^d$ ($d\geq 4$) is an open bounded domain, $f:\Omega\times\mathbb{R}\to\mathbb{R}$ is Carath\'{e}odory function, and $a, b, \lambda$ are positive parameters. First,  we have the following result concerning the Poisson-type problem (for simplicity, $\lambda=1$): 

\begin{theorem}\label{Poisson}
	Let $\Omega\subset \mathbb{R}^d$ be an open bounded domain $(d\geq 4)$,  $a,b>0$ be fixed numbers. 
	Then \begin{itemize}
		\item[(i)]  if $a^{\frac{d-4}{2}}b\geq \ds\mathtt{L}_d$, problem $(\mathscr{P}_{a,b}^h)$ has at least a  weak solution in $H^1_0(\Omega);$
		\item[(ii)] if $\ds a^{\frac{d-4}{2}}b> \mathtt{C}_d$, problem $(\mathscr{P}_{a,b}^h)$ has a unique weak solution in $H^1_0(\Omega)$.
	\end{itemize}
\end{theorem}

\noindent As a second application we consider the double-perturbed problem of $(\mathscr{P}_{a,b})$ of the form
$$
\left\{
\begin{array}{ll}
{\displaystyle -\left(a+b\int_{\Omega}|\nabla u|^{2}\right)\Delta u=|u|^{2^{*}-2}u}+\lambda |u|^{{p-2}}u+\mu g(x,u), & \hbox{ in } \Omega \\

u=0, & \hbox{on } \partial \Omega
\end{array}
\right. \eqno{(\mathscr{P}_{a,b}^{p,g})}
$$
where  $a, b, \lambda,\mu$ are positive parameters, $1<p<2^*$ and  $g:\Omega\times\mathbb{R}\to\mathbb{R}$ is Carath\'{e}odory function belonging to the class $\mathcal A$ which contains 
functions $\varphi:\Omega\times \mathbb{R}\to\mathbb{R}$ such that
$$\sup_{(x,t)\in \Omega \times \mathbb{R}}\frac{|\varphi(x,t)|}{1+|t|^{q-1}}<+\infty$$ for some  $1<q<2^*$. 

The single-perturbed problem  $(\mathscr{P}_{a,b}^{p,0})$ (i.e., $g\equiv 0$) is of particular interest. Indeed, when $\lambda>0$ is small enough and $d=4$,  Naimen \cite{N0} proved that  $(\mathscr{P}_{a,b}^{p,0})$ has a positive solution if and only if $b<\mathtt S_4^{-2} $; when $d>4$, there are also some sufficient conditions for guaranteeing the existence of positive solutions for $(\mathscr{P}_{a,b}^{p,0})$. In addition, if $p\in (2,4)$ and  $b>\mathtt S_4^{-2}= \mathtt{L}_4=\mathtt{PS}_4$, one can easily prove that any weak solution $u\in H_0^1(\Omega)$ of $(\mathscr{P}_{a,b}^{p,0})$ fulfills the \textit{a priori} estimate 
\begin{equation}\label{apriori}
\|u\|\leq \left(\frac{\lambda \mathtt S_4^{2-\frac{p}{2}}|\Omega|^{1-\frac{p}{2^*}}}{b\mathtt S_4^2- 1}\right)^\frac{1}{4-p}.
\end{equation}

% clearly, our  assumptions $b> \mathtt{L}_4=\mathtt{PS}_4$ complements our which is a kind of complementary of our assumption  .  Before to state our result,

The following result is twofold. First, it complements the result of Naimen \cite{N0} (i.e., we consider $b> \mathtt S_4^{-2}$ for $d=4$); second, having in our mind the global estimate (\ref{apriori}), 
%it shows the stability of $H_0^1$-norms of weak solutions of $(\mathscr{P}_{a,b}^{p,g})$ whenever $g$ is any subcritical function  (and $\lambda>0$ is large enough).  
it shows that the weak solutions of the perturbed problem  $(\mathscr{P}_{a,b}^{p,0})$ by means of any  subcritical function will be stable with respect to the $H_0^1$-norm (whenever $\lambda>0$ is large enough).

\begin{theorem}\label{japan} Let $\Omega\subset \mathbb{R}^d$ be an open bounded domain $(d\geq 4)$,  $a,b>0$ two fixed  numbers such that  $a^{\frac{d-4}{2}}b>\mathtt{PS}_d$ and $p\in (2,2^*)$.  Then there exists $\lambda^*>0$ such that  for each compact
	interval $[\alpha, \beta]\subset (\lambda^*, +\infty )$, there exists $r>0$
	with the following property: for every $\lambda \in [\alpha, \beta],$  and for every $g\in\mathcal A$,  there
	exists $\mu^* > 0$ such that for each $\mu \in [0, \mu^*],$ 
	problem $(\mathcal{P}_{\lambda,\mu}^{p,g})$ has at least three weak
	solutions   whose
	norms are less than $r$.
\end{theorem}

In fact, instead of Theorem \ref{japan} a slightly more general result will be given in  Section  \ref{section-appl}, replacing the term $u\mapsto|u|^{p-2}u$ by a function $f\in \mathcal A$ verifying some mild hypotheses. 

%
%
%
%In Section \ref{proofs} we are going to prove Theorem \ref{main1}, while in Section  \ref{section-appl} we provide two applications in the theory of elliptic PDEs  involving critical Kirchhoff terms. % (Poisson-type equation and sublinear  problem, respectively). 

\section{Proof of  Theorem \ref{main1}}\label{proofs}

\begin{proof}[Proof of Theorem {\rm\ref{main1}/(i)}] We divide the proof into two parts. 
	
	\textit{Step 1.}  Assume first that $a^{\frac{d-4}{2}}b\geq \ds\mathtt{L}_d$; we are going to prove that the energy functional $\mathcal{E}_{a,b}$ is sequentially weakly lower semicontinuous on  $H^1_0(\Omega)$.
	To see this, let $u \in H^1_0(\Omega)$ be arbitrarily fixed and consider a sequence $\{u_n\} \subset H^1_0(\Omega)$ such that $u_n\rightharpoonup u$ in $H^1_0(\Omega)$.
Thus, up to a subsequence, we have for every $p<2^*$ that $u_n\to u \hbox{ in }L^p(\Omega)$ and $\nabla u_n\rightharpoonup\nabla u \hbox{ in } L^2(\Omega)$ as $n\to \infty.$

By the latter relation, it is clear that \begin{align*}\|u_n\|^2-\|u\|^2=\|u_n-u\|^2+2\int_{\Omega}\nabla(u_n-u)\nabla u= \|u_n-u\|^2+o(1),\ \ {\rm as}\ n\to \infty.
\end{align*}
We also have that 
\begin{align*}
\|u_n\|^4-\|u\|^4&=\left(\|u_n\|^2-\|u\|^2\right)\left(\|u_n\|^2+\|u\|^2\right)\\&=\left(\|u_n-u\|^2+o(1)\right)\left(\|u_n-u\|^2+2\int_{\Omega}\nabla (u_n-u)\nabla u+2\|u\|^2\right)\\
&=\left(\|u_n-u\|^2+o(1)\right)\left(\|u_n-u\|^2+2\|u\|^2+o(1)\right),\ \ {\rm as}\ n\to \infty.
\end{align*}
On the other hand, by the Br\'ezis-Lieb Lemma (see e.g. Willem \cite{Willem}), one has 
 $$\|u_n\|_{2^*}^{2^*}-\|u\|_{2^*}^{2^*}=\|u_n-u\|_{2^*}^{2^*}+o(1),\ \ {\rm as}\ n\to \infty.$$
Combining the above estimates, it yields 
\begin{align*}
\mathcal{E}_{a,b}(u_n)-\mathcal{E}_{a,b}(u)=&\frac{a}{2}(\|u_n\|^2-\|u\|^2)+\frac{b}{4}(\|u_n\|^4-\|u\|^4)-\frac{1}{2^*}\left(\|u_n\|_{2^*}^{2^*}-\|u\|_{2^*}^{2^*}\right)\\=&\frac{a}{2}\|u_n-u\|^2+\frac{b}{4}\left(\|u_n-u\|^4+2\|u\|^2\|u_n-u\|^2\right)-\frac{1}{2^*}\|u_n-u\|_{2^*}^{2^*}+o(1) \\\overset{\eqref{Sob}}{\geq}& \frac{a}{2}\|u_n-u\|^2+\frac{b}{4}\left(\|u_n-u\|^4+2\|u\|^2\|u_n-u\|^2\right)-\frac{\mathtt{S}_d^{-\frac{2^*}{2}}}{2^*}\|u_n-u\|^{2^*}+o(1)
\\ \geq& \frac{a}{2}\|u_n-u\|^2 +\frac{b}{4}\|u_n-u\|^4-\frac{\mathtt{S}_d^{-\frac{2^*}{2}}}{2^*}\|u_n-u\|^{2^*}+o(1)\\=& \|u_n-u\|^2 \left(\frac{a}{2}+\frac{b}{4}\|u_n-u\|^2-\frac{\mathtt{S}_d^{-\frac{2^*}{2}}}{2^*}\|u_n-u\|^{2^*-2}\right)+o(1),\ \ {\rm as}\ n\to \infty.
\end{align*}
Let us consider the  function $f_d:[0,\infty)\to \mathbb R$ defined by \begin{equation}\label{fontosfuggveny}
f_d(x)=\frac{a}{2}+\frac{b}{4}x^2-\frac{\mathtt{S}_d^{-\frac{2^*}{2}}}{2^*}x^{2^*-2}, \ \ x\geq 0.
\end{equation} 
We claim that the function $f_d$ is positive for all $x\geq 0$.

\textit{Case 1}: $d=4$.  If follows that $2^*=4$, thus by the hypothesis $b\geq \ds\mathtt{L}_d$ -- which is equivalent to $b\mathtt{S}_4^2\geq 1$, -- it directly follows that
\begin{align*} f_4(x)&=\frac{a}{2}+\frac{b-\mathtt{S}_4^{-2}}{4}x^2\geq 0,\ \forall x\geq 0.
\end{align*}

\textit{Case 2}: $d>4$. The minimum of the function $f_d$ is at $m_d>0$, where $$m_d=\left(\frac{2^*b}{2(2^*-2)}\mathtt{S}_d^{\frac{2^*}{2}}\right)^{\frac{1}{2^*-4}}.$$ A simple algebraic computation shows that 
\begin{equation}\label{ekvivalencia-lsw}
\displaystyle a^\frac{d-4}{2}b\geq \mathtt{L}_d \Longleftrightarrow f_d(m_d)=\frac{1}{2}\left(a-b^{-\frac{2}{d-4}}\mathtt{L}_d^{\frac{2}{d-4}}\right)\geq 0,
\end{equation}
 which proves the claim.

Summing up the above estimates, we have that \begin{equation}
\label{later}
\liminf_{n\to \infty}(\mathcal{E}_{a,b}(u_n)-\mathcal{E}_{a,b}(u))\geq \liminf_{n \to \infty}\|u_n-u\|^2 f_d(\|u_n-u\|)\geq 0,\end{equation} which proves the sequentially weakly lower semicontinuity of  $\mathcal{E}_{a,b}$ on $H_0^1(\Omega)$.

	\textit{Step 2.}
Now, we prove that the constant $\mathtt{L}_d$ in Theorem \ref{main1} is sharp.  Assume the contrary, i.e., $\mathcal{E}_{a,b}$ is still  sequentially weakly lower semicontinuous   on $H_0^1(\Omega)$ for some $a,b>0$ with the property that 
\begin{equation}\label{a-b-tagadas}
a^{\frac{d-4}{2}}b<\mathtt{L}_d.
\end{equation}
%\begin{theorem}\label{main2}The following statements hold true:
%	\begin{enumerate}
%		\item[(a)] Assume that $b\mathtt{S}_4^2<1$. Then the energy functional $\mathcal{E}$ is not sequentially weakly lower semicontinuous.
%		\item[(b)] Assume that $\displaystyle a^\frac{d-4}{2}b<\mathtt{C}_d$. Then the energy functional $\mathcal{E}$ is not sequentially weakly lower semicontinuous.
%	\end{enumerate}
%\end{theorem}

%\begin{proof}[Proof of Theorem \ref{main2}] 
	
\textit{Case 1}: $d=4$. Fix  a minimizing sequence $\{u_n\}\subset H_0^1(\Omega)$ for $\mathtt{S}_4$ in \eqref{defSobolev}; by its boundedness it is clear that there exists $u\in H_0^1(\Omega)\setminus \{0\}$ such that, up to a subsequence, $u_n\rightharpoonup u$ in $H^1_0(\Omega)$. Moreover,  the sequentially weakly lower semicontinuity of the norm $\|\cdot \|$ implies that $\ds \|u\|\leq \liminf_{n\to \infty}\|u_{n}\|=:L$ and there exists a subsequence $\{u_{n_j}\}$ of $\{u_n\}$ such that  $\ds L= \lim_{j\to \infty}\|u_{n_j}\|;$  in particular, $L>0.$

By recalling the function $f_4$ from (\ref{fontosfuggveny}), due to (\ref{a-b-tagadas}), it is clear that on $(x_0,\infty)$ the function $f_4$ is decreasing and negative, where $\displaystyle x_0=\left(\frac{2a\mathtt{S}_4^2}{1-\mathtt{S}_4^2b}\right)^\frac{1}{2}$ is the unique solution of $f_4(x)=0$, $x\geq 0$.
\begin{figure}[H]
	\centering
	\label{abra1}
	\includegraphics[scale=0.35]{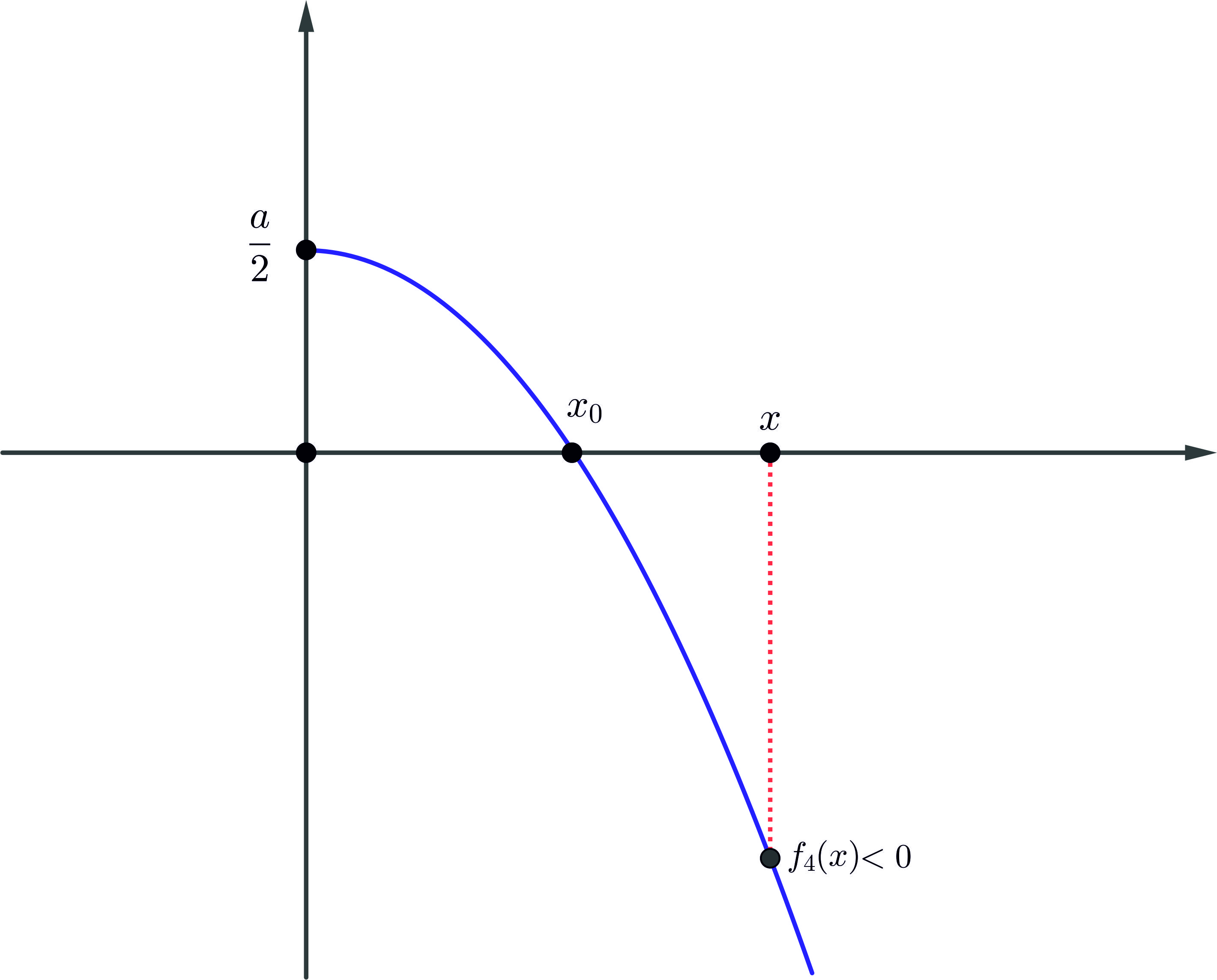}
	\caption{Shape of the function $\ds x\mapsto f_4(x)$, $x\geq 0,$ when (\ref{a-b-tagadas}) holds.}
\end{figure}
Let $c>0$ be such that $cL\geq c\|u\|>x_0$. It is also clear that $\{cu_{n_j}\}$ is a minimizing sequence for $\mathtt{S}_4$ and $cu_{n_j}\rightharpoonup cu$ in $H_0^1(\Omega)$ as $j\to \infty$. Consequently, since $f_4$ is continuous, we have that 
 \begin{eqnarray}\nonumber\liminf_{n\to \infty}\mathcal{E}_{a,b}(cu_n)&\leq &  \liminf_{j\to \infty}\mathcal{E}_{a,b}(cu_{n_j})\\ &=&\liminf_{j\to \infty}\left\{\frac{a}{2}\|cu_{n_j}\|^2+\frac{b}{4}\|cu_{n_j}\|^4-\frac{\mathtt{S}_4^{-2}}{4}\|cu_{n_j}\|^4\right\}\nonumber\\
&=&\liminf_{j\to \infty} \|cu_{n_j}\|^2 f_4\left(\|cu_{n_j}\|\right)\nonumber\\&=& (cL)^2 f_4(cL)\label{fontos}.
\end{eqnarray}
Since  $cL\geq c\|u\|>x_0$, we have that $f_4(cL)\leq f_4(\|cu\|)<0$, thus by \eqref{fontos}, we get that $$\liminf_{n\to \infty}\mathcal{E}_{a,b}(cu_n)\leq \|cu\|^2f_4(\|cu\|).$$ On the other hand, by (\ref{Sob}) we have  \begin{align}\label{meg-ez-is-kell}\|cu\|^2f_4(\|cu\|)&=\frac{a}{2}\|cu\|^2+\frac{b}{4}\|cu\|^4-\frac{\mathtt{S}_4^{-2}}{4}\|cu\|^4\nonumber \\ &\leq \nonumber \frac{a}{2}\|cu\|^2+\frac{b}{4}\|cu\|^4-\frac{1}{4}\int_{\Omega}|cu|^4\\ &= \mathcal{E}_{a,b}(cu).
\end{align}
By the above estimates we have that 
$\liminf_{n\to \infty}\mathcal{E}_{a,b}(cu_n)\leq \mathcal{E}_{a,b}(cu)$. 
In fact, we have strict inequality in the latter relation; indeed, otherwise we would have $\mathtt{S}_4^{-2}\|cu\|^4=\|cu\|_{4}^{4},$ i.e., $u$ would be an extremal function in \eqref{defSobolev}. However, since $\Omega\neq \mathbb{R}^d$, no extremal function exists in  \eqref{defSobolev}, see Willem  \cite[Proposition 1.43]{Willem}. Thus, we indeed have  $$\liminf_{n \to \infty}\mathcal{E}_{a,b}(cu_n)<\mathcal{E}_{a,b}(cu),$$ which contradicts the  sequentially weakly lower semicontinuouity of $\mathcal{E}_{a,b}$  on $H_0^1(\Omega)$. Accordingly,  it yields that (\ref{a-b-tagadas}) cannot hold whenever the sequentially weakly lower semicontinuouity of $\mathcal{E}_{a,b}$  on $H_0^1(\Omega)$ is assumed,  which proves the optimality of the constant $\mathtt{L}_d$ in the case when $d=4$.

\medskip

\textit{Case 2}: $d>4$. Since $0<2^*-2<2,$ it is clear that $f_d(+\infty)=+\infty$ and  the assumption (\ref{a-b-tagadas}) together with the equivalence (\ref{ekvivalencia-lsw}) ensures that the function $f_d$  has its global minimum point at $m_d>0$ with $f_d(m_d)<0$.

\begin{figure}[H]
	\centering
	\label{abra2}
	\includegraphics[scale=0.30]{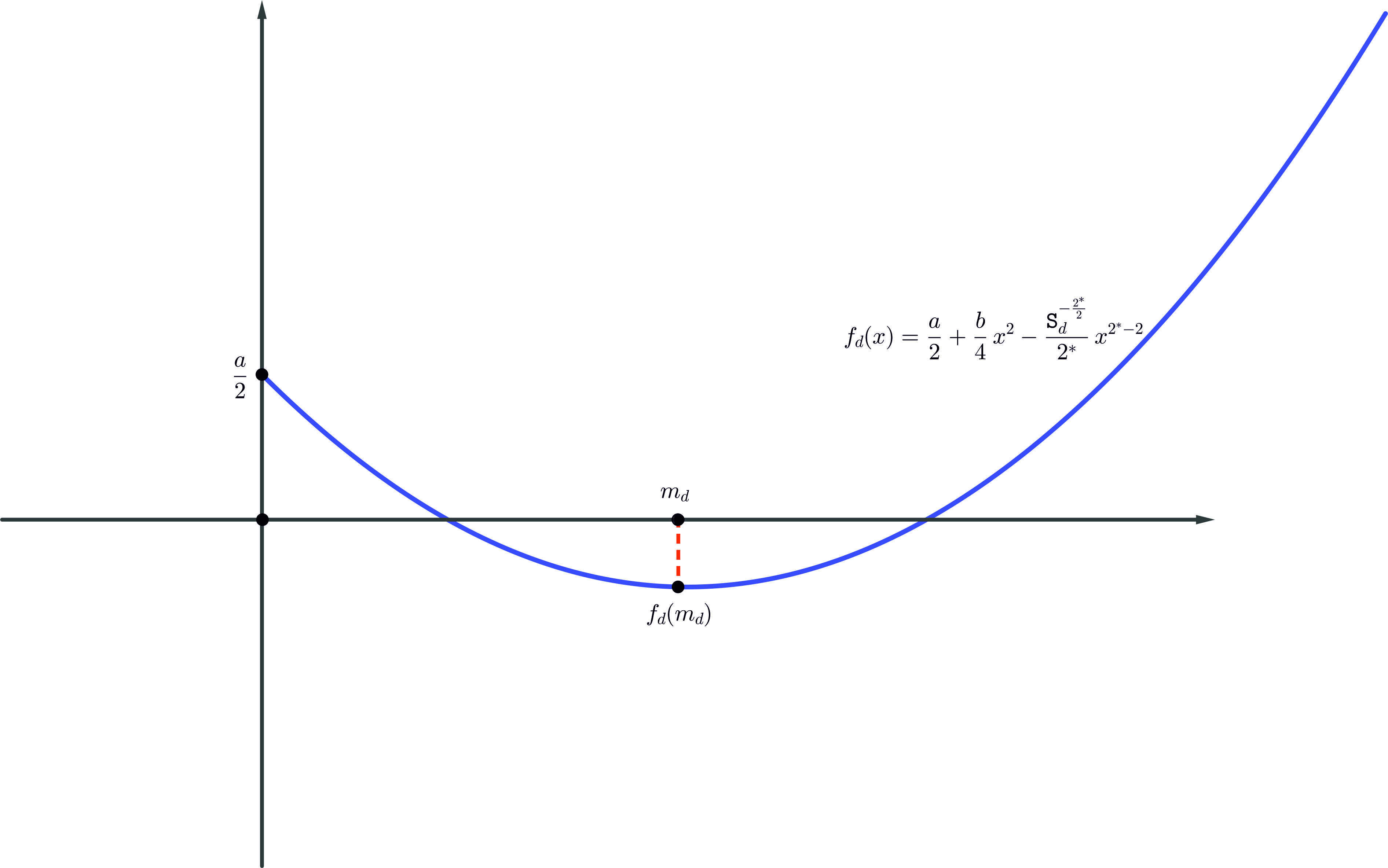}
	\caption{Shape of the function $\ds x\mapsto f_d(x)$, $x\geq 0,$ when $d>4$ and (\ref{a-b-tagadas}) holds.}
\end{figure}
Consider a minimizing sequence $\{u_n\}\subset H_0^1(\Omega)$ for $\mathtt{S}_d$, and let $\{u_{n_j}\}$ be a subsequence of $\{u_n\}$,  $L>0$ and $u\in H_0^1(\Omega)$ as in \textit{Case 1}.   Let $c=\frac{m_d}{L}>0$. Since $\|u\|\leq L$ and the minimum has the property that $f_d(m_d)<0,$ it follows, similarly as before, that 
 \begin{eqnarray*}\nonumber\liminf_{n\to \infty}\mathcal{E}_{a,b}(cu_n)&\leq &  \liminf_{j\to \infty}\mathcal{E}_{a,b}(cu_{n_j})\\ &=&\liminf_{j\to \infty}\left\{\frac{a}{2}\|cu_{n_j}\|^2+\frac{b}{4}\|cu_{n_j}\|^4-\frac{\mathtt{S}_d^{-\frac{2^*}{2}}}{2^*}\|cu_{n_j}\|^{2^*}\right\}\nonumber\\
&=&\liminf_{j\to \infty} \|cu_{n_j}\|^2 f_d\left(\|cu_{n_j}\|\right)\nonumber\\&=& (cL)^2 f_d(cL)=(cL)^2 f_d(m_d)\\&\leq&\|cu\|^2 f_d(m_d)\\&\leq & \|cu\|^2 f_d(\|cu\|). 
\end{eqnarray*}
Similarly as in (\ref{meg-ez-is-kell}) and using  Willem \cite[Proposition 1.43]{Willem}, we have that $\|cu\|^2f_d(\|cu\|)<\mathcal{E}_{a,b}(cu),$ i.e., $\mathcal{E}_{a,b}$ is not sequentially weakly lower semicontinuous  on $H_0^1(\Omega)$, a contradiction. 
\end{proof}

\begin{proof}[Proof of  Theorem {\rm\ref{main1}/(ii)}]
	Let $\{u_{n}\}\subset H_0^1(\Omega)$ be a  PS-sequence for $\mathcal E_{a,b}$, i.e., for some $c\in \mathbb R$, 
	\[
	\begin{cases}
	\mathcal{E}_{a,b}(u_{n})\to c\\
	\mathcal{E}_{a,b}'(u_{n})\to0
	\end{cases}\mbox{as }n\to\infty.
	\]
One can prove that $\mathcal{E}_{a,b}$ is of class $C^2$ on $H_0^1(\Omega)$; in particular,  a direct calculation yields  (see also Willem \cite[Proposition 1.12]{Willem}) that 
\begin{equation}\label{E-derivalt}
\langle\mathcal{E}_{a,b}'(u),v\rangle=\left(a+b\|u\|^2\right)\int_\Omega \nabla u\nabla v -\int_\Omega |u|^{2^*-2}uv ,\ \forall u,v\in H_0^1(\Omega).
\end{equation}
Note that $\mathcal E_{a,b}$ is coercive on $H_0^1(\Omega)$; indeed, the claim follows by (\ref{Sob}) together with the facts that if $d>4$ then $4>2^*$, while if $d=4$ then   $b>\mathtt{PS}_{4}=\mathtt{S}_{4}^{-2}$.  In particular,  it follows that  $\{u_{n}\}$ is bounded in $H_0^1(\Omega)$, thus there exists $u\in H_{0}^{1}(\Omega)$ such that (up to a subsequence),
	\begin{align*}
	u_{n} & \rightharpoonup u\mbox{ in }H_{0}^{1}(\Omega),\\
	u_{n} & \to u\mbox{ in }L^{p}(\Omega),\ p\in[1,2^{*}),\\
	u_{n} & \to u\mbox{ a.e. in }\Omega.
	\end{align*}
	By using the second concentration compactness lemma of Lions \cite{L}, there exist an at most countable index set $J$,
	a set of points $\{x_{j}\}_{j\in J}\subset\Omega$ and two families of positive
	numbers $\{\eta_{j}\}_{j\in J}$, $\{\nu_{j}\}_{j\in J}$ such that
	\begin{align}
	|\nabla u_{n}|^{2} & \rightharpoonup d\eta\geq|\nabla u|^{2}+\sum_{j\in J}\eta_{j}\delta_{x_{j}},\label{11-}\\
	|u_{n}|^{2^{*}} & \to d\nu=|u|^{2^{*}}+\sum_{j\in J}\nu_{j}\delta_{x_{j}},\label{22-}
	\end{align}
	in the sense of measures, where $\delta_{x_{j}}$ is the Dirac mass concentrated at
	$x_{j}$ and such that
\begin{equation}\label{munu}
	\mathtt{S}_{d}  \nu_{j}^{\frac{2}{2^{*}}}\leq\eta_{j},\ \forall j\in J.
\end{equation}

We are going to prove that the index set $J$ is empty. Arguing
	by contradiction, we may assume that there exists a $j_{0}$ such
	that $\nu_{j_{0}}\neq0$ at $x_0$. For a sufficiently small $\varepsilon>0$ we consider a non-negative  cut-off function $\phi_\varepsilon$ such that
	\begin{align*}
	&0\leq \phi_{\varepsilon}  \leq 1\mbox{ in }\Omega,\\
	&\phi_{\varepsilon}  \equiv1\mbox{ in }B(x_{0},\varepsilon),\\
	&\phi_{\varepsilon}  =0\mbox{ in } \Omega\setminus B(x_{0},2\varepsilon),\\
	&|\nabla\phi_{\varepsilon}|  \leq\frac{2}{\varepsilon},
	\end{align*}
where $B(x_0,r)=\{x\in \mathbb R^d:|x-x_0|<r \}$ for $r>0.$	It is clear that the sequence $\{u_{n}\phi_{\varepsilon}\}$ is
	bounded in $H_{0}^{1}(\Omega)$,  thus 
	\[
	\lim_{n\to\infty}\mathcal{E}_{a,b}'(u_{n})(u_{n}\phi_{\varepsilon})=0.
	\]
	In particular, by (\ref{E-derivalt})  it turns out that when $n\to \infty$, one has
	\begin{eqnarray*}
	o(1) &=& \mathcal{E}_{a,b}'(u_{n})(u_{n}\phi_{\varepsilon})\\&=&(a+b\|u_{n}\|^{2})\int_{\Omega}\nabla u_{n}\nabla(u_{n}\phi_{\varepsilon})-\int_{\Omega}|u_{n}|^{2^{*}}\phi_{\varepsilon}\\
	& =&(a+b\|u_{n}\|^{2})\left(\int_{\Omega}|\nabla u_{n}|^{2}\phi_{\varepsilon}+\int_{\Omega}u_{n}\nabla u_{n}\nabla\phi_{\varepsilon}\right)-\int_{\Omega}|u_{n}|^{2^{*}}\phi_{\varepsilon}.
	\end{eqnarray*}

First, by  H\"{o}lder's inequality, there exists $C>0$ (not depending on $n$) such that
	
	\begin{eqnarray*}
		\left|\int_\Omega u_{n}\nabla u_{n}\nabla\phi_{\varepsilon}\right|&=&\left|\int_{B(x_{0},2\varepsilon)}u_{n}\nabla u_{n}\nabla\phi_{\varepsilon}\right|\leq \left(\int_{B(x_{0},2\varepsilon)}|\nabla u_n|^2\right)^\frac{1}{2}
		\left(\int_{B(x_{0},2\varepsilon)}|u_n\nabla \phi_\varepsilon|^2\right)^\frac{1}{2}\\
		&\leq&  C \left(\int_{B(x_{0},2\varepsilon)}|u_n\nabla \phi_\varepsilon|^2\right)^\frac{1}{2}.
	\end{eqnarray*}
The Lebesgue dominated convergence theorem implies that $$\lim_{n\to\infty}\int_{B(x_{0},2\varepsilon)}|u_n\nabla \phi_\varepsilon|^2=\int_{B(x_{0},2\varepsilon)}|u\nabla \phi_\varepsilon|^2,$$ and by
	\begin{eqnarray*}
		\left(\int_{B(x_{0},2\varepsilon)}|u\nabla \phi_\varepsilon|^2\right)^\frac{1}{2}&\leq &
		\left(\int_{B(x_{0},2\varepsilon)} |u|^{2^*}\right)^\frac{1}{2^*}
		\left(\int_{B(x_{0},2\varepsilon)}|\nabla \phi_\varepsilon|^d \right)^\frac{1}{d}\\
		&\leq& C \left(\int_{B(x_{0},2\varepsilon)} |u|^{2^*}\right)^\frac{1}{2^*},
	\end{eqnarray*}
	for some $C>0$, we obtain
	
	\[
	\lim_{\varepsilon\to0}\lim_{n\to\infty}(a+b\|u_{n}\|^{2})\left|\int_\Omega u_{n}\nabla u_{n}\nabla\phi_{\varepsilon}\right|=0.
	\]
	Second, by (\ref{11-}) it follows that 
	\begin{eqnarray*}
		\lim_{n\to\infty}(a+b\|u_{n}\|^{2})\int_{\Omega}|\nabla u_{n}|^{2}\phi_{\varepsilon}&\geq&
		\lim_{n\to\infty}\left[a\int_{B(x_{0},2\varepsilon)}|\nabla u_{n}|^{2}\phi_{\varepsilon}+b\left(\int_{\Omega}|\nabla u_{n}|^{2}\phi_{\varepsilon}\right)^{2}\right]\\&\geq&
		a\int_{B(x_{0},2\varepsilon)}|\nabla u|^{2}\phi_{\varepsilon}+b\left(\int_{\Omega}|\nabla u|^{2}\phi_{\varepsilon}\right)^{2}+a\eta_{j_{0}}+b\eta_{j_{0}}^{2},
	\end{eqnarray*}
	thus
	\[
	\lim_{\varepsilon\to0}\lim_{n\to\infty}(a+b\|u_{n}\|^{2})\int_{\Omega}|\nabla u_{n}|^{2}\phi_{\varepsilon} \geq a\eta_{j_{0}}+b\eta_{j_{0}}^{2}.\]
	Third, by (\ref{22-}) one has that 
	\begin{align*}
	\lim_{\varepsilon\to0}\lim_{n\to\infty}\int_\Omega|u_{n}|^{2^{*}}\phi_{\varepsilon} & =\lim_{\varepsilon\to0}\int_\Omega |u|^{2^{*}}\phi_{\varepsilon}+\nu_{j_{0}}=\lim_{\varepsilon\to0}\int_{B(x_{0},2\varepsilon)} |u|^{2^{*}}\phi_{\varepsilon}+\nu_{j_{0}}=\nu_{j_{0}}.
	\end{align*}

	Summing up the above estimates, one obtains
	\begin{align}\label{33-}
	0 & \geq a\eta_{j_{0}}+b\eta_{j_{0}}^{2}-\nu_{j_0}\geq a\eta_{j_{0}}+b\eta_{j_{0}}^{2}-\mathtt{S}_{d}^{-\frac{2^{*}}{2}}\eta_{j_{0}}^{\frac{2^{*}}{2}}\nonumber \\
	& =\eta_{j_{0}}\left(a+b\eta_{j_{0}}-\mathtt{S}_{d}^{-\frac{2^{*}}{2}}\eta_{j_{0}}^{\frac{2^{*}}{2}-1}\right).
	\end{align}
	Let $\tilde f_d:[0,\infty)\to\mathbb{R}$ be the function defined by $${\displaystyle \tilde f_d(x)=a+bx-\mathtt{S}_{d}^{-\frac{2^{*}}{2}}x^{\frac{2^{*}}{2}-1}},\ x\geq 0.$$
	One can see that the assumption   $\ds a^{\frac{d-4}{2}}b>\mathtt{PS}_{d}$ implies that 
	$\tilde f_d(x)>0$ for all $x\geq0$. In particular, it follows that
	$
	a+b\eta_{j_{0}}-\mathtt{S}_{d}^{-\frac{2^{*}}{2}}\eta_{j_{0}}^{\frac{2^{*}}{2}-1}>0,
	$
	therefore by (\ref{33-}) we necessarily have  that $\eta_{j_{0}}=0,$ contradicting $\nu_{j_0}\neq 0$ and (\ref{munu}). The latter fact  implies
	that $J$ is empty. In particular, by (\ref{22-}) and  Brezis-Lieb lemma it follows that $u_{n}\to u\mbox{ in }L^{2^{*}}(\Omega)$ as $n\to \infty$; thus
	\begin{equation}\label{bl-krit}
	\lim_{n\to\infty}\ds \int_{\Omega}|u_n|^{2^*-2}u_n(u-u_n)=0.
	\end{equation}
	Since $\mathcal{E}_{a,b}'(u_{n})\to0$ as $n\to \infty$, we have  by (\ref{E-derivalt}) and (\ref{bl-krit}) that  
	\begin{eqnarray*}
	0&=&\lim_{n\to\infty}\mathcal{E}_{a,b}'(u_{n})(u_{n}-u)\\&=&\lim_{n\to\infty}\left((a+b\|u_{n}\|^{2})\int_{\Omega}\nabla u_n(\nabla u-\nabla u_n)+\int_{\Omega}|u_n|^{2^*-2}u_n(u-u_n)\right)\\&=&\lim_{n\to\infty}\left((a+b\|u_{n}\|^{2})\int_{\Omega}\nabla u_n\nabla (u- u_n)\right).
	\end{eqnarray*}
	By the boundedness of $\{u_n\}\subset H_0^1(\Omega)$, the latter relation and the fact that $u_n\rightharpoonup u$ in $H_0^1(\Omega)$, i.e., $\ds\int_{\Omega}\nabla u\nabla (u- u_n)\to 0$ as $n\to \infty$, we obtain at once that  $\|u_n-u\|^2\to 0$ as $n\to \infty$, which concludes the proof.
%	, thus by (\ref{11-}) and  (\ref{22-}) one has in particular that 
%	$\liminf_{n\to \infty}\|u_n\|\geq \|u\|$ and 
%	$
%	u_{n}\to u\mbox{ in }L^{2^{*}}(\Omega),
%$ respectively. 
%	Now, recalling that the derivative of  $$u\mapsto \frac{a}{2}\|u\|^{2}+\frac{b}{4}\|u\|^{4}$$  is an $(S_+)$-type operator, we have that 
%Thus, up to a subsequence, we have that
%%$$\|u_n-u\|^2=\|u_n\|^2-\|u\|^2-2\int \nabla u\nabla (u_n-u)\to 0,\ \ {\rm as}\ n\to \infty,$$ 
%$u_{n}\to u\mbox{ in }H_{0}^{1}(\Omega)$, which concludes the proof. 
\end{proof}

\begin{proof}[Proof of  Theorem {\rm\ref{main1}/(iii)}]
	It is well known that the energy functional $\mathcal{E}_{a,b}:H^1_0(\Omega)\to \mathbb{R}$ is convex if and only id $\mathcal{E}_{a,b}'$ is monotone, or equivalently, $$\langle\mathcal{E}_{a,b}''(u)v,v\rangle\geq 0,\ \forall u,v\in H^1_0(\Omega).$$
By using (\ref{E-derivalt}), we have 
	$$\langle\mathcal{E}_{a,b}''(u)v,v\rangle=a\|v\|^2+b\|u\|^2\|v\|^2+2b\left(\int_\Omega\nabla u\nabla v \right)^2-(2^*-1)\int_\Omega |u|^{2^*-2}v^2 .$$
	Moreover, by H\"older and Sobolev inequalities, one can see that 
	\begin{align*}\langle\mathcal{E}_{a,b}''(u)v,v\rangle &\geq a\|v\|^2+b\|u\|^2\|v\|^2-(2^*-1)\mathtt{S}_d^{-\frac{2^*}{2}}\|u\|^{2^*-2}\|v\|^2\\ &=\|v\|^2 \left[a+b\|u\|^2-(2^*-1)\mathtt{S}_d^{-\frac{2^*}{2}}\|u\|^{2^*-2}\right].
	\end{align*}
	
	Let us consider  the function $\overline f_d:[0,\infty)\to \mathbb R$ given by $$\overline f_d(x)=a+bx^2-(2^*-1)\mathtt{S}_d^{-\frac{2^*}{2}}x^{2^*-2},\ x\geq 0.$$
	
	We claim that the function $\overline f_d$ is positive on $[0,\infty)$.
	
	\textit{Case 1}: $d=4$.  Since $2^*=4$, the hypothesis $b\geq \ds\mathtt{C}_d$ (which is equivalent to $b\mathtt{S}_4^2\geq 3$) implies that
	\begin{align*} \overline f_4(x)&=a+bx^2-3\mathtt{S}_4^{-2}x^2=a+x^2\left(b-3\mathtt{S}_4^{-2}\right)\geq 0,\ \forall x\geq 0.
	\end{align*}
	
	\textit{Case 2}: $d>4$. The global minimum of the function $\overline f_d$ is at $\overline m_d>0$, where $$\overline m_d=\left[\frac{2b\mathtt{S}_d^{\frac{2^*}{2}}}{(2^*-1)(2^*-2)}\right]^{\frac{1}{2^*-4}}.$$ It turns out that  
$$
	\displaystyle a^\frac{d-4}{2}b\geq \mathtt{C}_d \Longleftrightarrow \overline f_d(\overline m_d)\geq 0,
$$
	which proves the claim. The strict convexity of $\mathcal{E}_{a,b}$ similarly follows whenever $a^\frac{d-4}{2}b> \mathtt{C}_d $ is assumed. 
\end{proof}

\section{Applications: proof of Theorems \ref{Poisson}\&\ref{japan}}\label{section-appl}
%\subsection{Poisson-type perturbation} As a simple application of Theorem \ref{main1}/(i) and (iii) one can consider the model Poisson type problem
%
%\[
%\ \left\{ \begin{array}{lll}
%{\displaystyle -\left(a+b\int_{\Omega}|\nabla u|^{2}\right)\Delta u=|u|^{2^{*}-2}u+h(x)} & \mbox{in} & \Omega,\\
%u=0 & \mbox{on} & \partial\Omega,
%\end{array}\right.\eqno{(\mathscr{P}_{a,b}^h)}
%\]
%where $h\in L^\infty(\Omega)$ is a positive function.

\begin{proof}[Proof of Theorem \ref{Poisson}]
		We consider the energy functional associated with problem $(\mathscr{P}_{a,b}^h)$, i.e.,
		\begin{equation*}
		\mathcal{E}(u)=\mathcal{E}_{a,b}(u)-\int_{\Omega
		}h (x)u(x)dx,\ \ u\in
		H_{0}^1(\Omega).
		\end{equation*}%
		It is easy to prove that  $\mathcal{E}$ belongs to $C^1(H_0^1(\Omega),\mathbb{R}%
		)$ and its critical points are exactly the weak solutions of problem $(\mathscr{P}_{a,b}^h)$. 
	Moreover, $\mathcal{E}$ is bounded from below and coercive on $H_0^1(\Omega),$ i.e., $\mathcal{E}(u)\rightarrow
		+\infty $ whenever
		$\|u\|\rightarrow +\infty $. 
		
		(i) If $a^{\frac{d-4}{2}}b\geq \ds\mathtt{L}_d$, by Theorem \ref{main1}/(i) and the fact that $u\mapsto \ds \int_{\Omega}h(x)u(x)dx$ is sequentially weakly continuous on $H_0^1(\Omega)$ (due to the boundedness of $\Omega$ and the compactness of the  embedding $H_0^1(\Omega)$ into $L^p(\Omega)$, $p\in [1,2^*)$),  $\mathcal{E}$ turns to be sequentially weakly lower semicontinuous on $H^1_0(\Omega)$. Thus the basic result of the
		calculus of variations
		implies that $\mathcal E$ has a global minimum point $u\in H_0^1(\Omega)$, see Zeidler \cite[Proposition 38.15]{Zeidler},  which is also a 
		 critical point of $\mathcal{E}.$
		
		(ii)  If $\ds a^{\frac{d-4}{2}}b> \mathtt{C}_d$, Theorem \ref{main1}/(iii)   implies that 
		$\mathcal{E}$ is strictly convex on $H_0^1(\Omega)$. By Zeidler \cite[
		Theorem 38.C]{Zeidler}  it follows that $\mathcal E$ has at most one minimum/critical point. The inequality (\ref{sorrend}) and (i) conclude the proof.  
\end{proof}

%\subsection{Subcritical perturbation.} 
%As a second application we consider the  problem
%
%$$
%\left\{
%\begin{array}{ll}
%{\displaystyle -\left(a+b\int_{\Omega}|\nabla u|^{2}\right)\Delta u=|u|^{2^{*}-2}u}+\lambda f(x,u)+\mu g(x,u), & \hbox{ in } \Omega \\
%
%u=0, & \hbox{on } \partial \Omega
%\end{array}
%\right. \eqno{(\mathscr{P}_{a,b}^{f,g})}
%$$
%where $\Omega\subset \mathbb{R}^d$ ($d\geq 4$) is an open bounded domain, $f, g:\Omega\times\mathbb{R}\to\mathbb{R}$ are Carath\'{e}odory functions, and $a, b, \lambda,\mu$ positive parameters.
%
%
%
%Denote by  $\mathcal A$  the set of Carath\'{e}odory
%functions $\varphi:\Omega\times \mathbb{R}\to\mathbb{R}$ such that
%$$\sup_{(x,t)\in \Omega \times \mathbb{R}}\frac{|\varphi(x,t)|}{1+|t|^{q-1}}<+\infty$$ for some  $1<q<2^*$. 

For $f\in \mathcal A$, let us denote by $\ds F(x,t)=\int_0^t f(x,s)ds.$ We now prove the following result which directly implies Theorem \ref{japan}.   

\begin{theorem}\label{utolso-tetel} Let $\Omega\subset \mathbb{R}^d$ be an open bounded domain $(d\geq 4)$, let $f\in \mathcal A$, and $a,b>0$ two fixed  numbers such that  $a^{\frac{d-4}{2}}b>\mathtt{PS}_d$. Assume also that
	\begin{itemize}
		\item[$H_1)$] $\ds\lim_{t \rightarrow 0}\frac{\ds\sup_{x\in\Omega}F(x,t)}{t^2}\leq 0;$
		
		\item[$H_2)$]  %$\ds\inf_{x\in \Omega}F(x,t_0)>0$ for some $t_0>0.$  
		$\ds \sup_{u\in \W}\int_\Omega F(x,u)>0$.
	\end{itemize}
	Set
	\begin{equation}\label{lambdastar}
	\lambda^* = \inf \left\{\frac{\ds\frac{a}{2}\|u\|^2+\frac{b}{4}\|u\|^4-\frac{1}{2^*}\|u\|_{2^*}^{2^*}}{\ds\int_\Omega F(x,u) } : u\in\W,  \int_\Omega F(x,u) >0\right\}.
	\end{equation}
	Then, for each compact
	interval $[\alpha, \beta]\subset (\lambda^*, +\infty )$, there exists $r>0$
	with the following property: for every $\lambda \in [\alpha, \beta],$  and for every $g\in\mathcal A$,  there
	exists $\mu^* > 0$ such that for each $\mu \in [0, \mu^*],$ the
	problem $(\mathcal{P}_{a,b}^{f,g})$ has at least three weak
	solutions   whose
	norms are less than $r$.
\end{theorem}
\begin{proof} 
	Denote by $J_f :\W\to\R$ the functional defined by
	\[J_f(u)=\int_\Omega F(x,u), \]
	and consider as before the functional \[\mathcal{E}_{a,b}(u)=\frac{a}{2}\|u\|^2+ \frac{b}{2}\|u\|^4-\frac{1}{2^*}\|u\|_{2^*}^{2^*} ,\]
	From Theorem \ref{main1}/(i), $\mathcal{E}_{a,b}$ is sequentially weakly lower semicontinuous (see \eqref{sorrend}), and if $\{u_n\}$ weakly converges to $u$ and $\ds \liminf_{n\to\infty}\mathcal{E}_{a,b}(u_n)\leq \mathcal{E}_{a,b}(u)$, then $\{u_n\}$ has a subsequence strongly convergent to $u$, see \eqref{later}.  Moreover it is of class $C^1$ on $\W$. Since $f$ has a subcritical growth, $J_f$ is sequentially weakly continuous in $\W$, of class $C^1$ too and bounded on bounded sets.
	
	From assumption $H_1)$ it follows that   
	\[\limsup_{u\to 0}\frac{J_f(u)}{\mathcal{E}_{a,b}(u)}\leq 0,\] therefore $\mathcal{E}_{a,b}-\lambda J_f$ has a (strong) local minimum at zero for every $\lambda>0$. By Ricceri \cite[Theorem C]{Ric1}, zero turns out to be  a local minimizer of $\mathcal{E}_{a,b}-\lambda J_f$ in the weak topology of $\W$.  It is also clear that  $\mathcal{E}_{a,b}-\lambda J_f$ is  coercive for every $\lambda$ and, if $\lambda>\lambda^*$, its global minimum  is different to zero. 
	
	To proceed, fix $[\alpha, \beta]\subset(\lambda^*,+\infty)$ and choose
	$\;\sigma>0$. By the coercivity of $\mathcal{E}_{a,b}-\lambda J_f$ it 
	follows that the set
	$\overline{(\mathcal{E}_{a,b}-\lambda J_f)^{-1}((-\infty,\sigma))}^w$
	is   compact and metrizable  with respect
	to the weak topology. Also,  
	\[\bigcup_{\lambda\in [\alpha,\beta]}(\mathcal{E}_{a,b}-\lambda J_f)^{-1}((-\infty,\sigma))\}\subseteq
	B_\eta,\]
	for some positive radius $\eta$, where $B_\eta=\{u\in H_0^1(\Omega):\|u\|<\eta\}$. Let $\ds c^*=\sup_{B_\eta}\mathcal{E}_{a,b}+ \beta\sup_{B_\eta}|J_f|$ and let
	$r>\eta$ be such that
	\begin{equation}\label{inclusion}\bigcup_{\lambda\in [\alpha,\beta]}(\mathcal{E}_{a,b}-\lambda J_f)^{-1}((-\infty, c^*+2])\}\subseteq B_r.\end{equation}
	
	Let $\lambda\in [\alpha,\beta]$ and fix $g\in \mathcal A$. Thus, if $J_g:\W\to\R$ is the functional defined by 
	
	\[J_g(u)=\int_\Omega G(x,u) \quad \mbox{where} \quad  G(x,t)=\int_0^t g(x,s)ds, \] then, $J_g$ is of class $C^1$, with compact derivative.
	Choose  a function $h\in C^1(\R)$, bounded, such that $h(t)=t$ for every $t$ such that $\ds |t|\leq \sup_{B_{r}}|J_g|$. Define $\tilde J_g =h\circ J_g$. Then, $\tilde J_g$ has compact derivative and $\tilde J_g(u)=J_g(u)$ for every $u\in B_r$.

	Applying Ricceri \cite[Theorem 4]{R} with $P=\mathcal{E}_{a,b}-\lambda J_f, \ Q=\tilde J_g$, $\tau$ the weak topology of $\W$,  we deduce the existence of some $\;\delta>0$
	such that for every $\mu\in[0,\delta]$, $\mathcal{E}_{a,b}-\lambda J_f-\mu \tilde J_g$ has two
	local minimizers in the $\tau_{\mathcal{E}_{a,b}-\lambda J_f}$ topology (the smallest topology containing both the weak topology and the sets $\{(\mathcal{E}_{a,b}-\lambda J_f)^{-1}((-\infty,s))\}_{s\in \mathbb R}$), say $u_1, u_2$, such
	that
	
	\begin{equation} \label{bound}
	u_1, u_2 \in (\mathcal{E}_{a,b}-\lambda J_f)^{-1}((-\infty,\sigma) )\; \subseteq B_{\eta}
	\;\subseteq\; B_r\;.
	\end{equation}
	
	Since
	the topology $\tau_{\mathcal{E}_{a,b}-\lambda J_f}$ is
	weaker than the strong topology, $u_1$ and $ u_2$ turn out to be
	local minimizers of the
	functional $\mathcal{E}_{a,b}-\lambda J_f-\mu \tilde J_g$. Define now $\ds\mu^*=\min\left\{\delta, \frac{1}{\sup_{\R}h}\right\}.$ One can see that $\mathcal{E}_{a,b}-\lambda J_f-\mu \tilde J_g$ satisfies the Palais-Smale condition as in the Theorem \ref{main1}/(ii) (Palais-Smale for $\mathcal{E}_{a,b}$), thus from Pucci and Serrin  \cite[Theorem 1]{PS} there exists a critical
	point of $\mathcal{E}_{a,b}-\lambda J_f-\mu \tilde J_g$, say $u_3,$ such that
	\[(\mathcal{E}_{a,b}-\lambda J_f-\mu \tilde J_g)(u_3)=\inf_{\gamma\in \mathcal S}\sup_{t\in[0,1]}(\mathcal{E}_{a,b}-\lambda J_f-\mu \tilde J_g)(\gamma(t)),\]
	where \[\mathcal S=\{\gamma\in C^0([0,1],\W): \ \gamma(0)=u_1,\
	\gamma(1)=u_2\}.\]

	In particular, if $\;\tilde \gamma(t)=t
	u_1+(1-t)u_2, \; t\in [0,1]$, then $\tilde\gamma\in \mathcal S$ and
	\[
	\tilde \gamma(t) \in B_{\eta}, \quad \mbox{for all}\;\;t\in [0,1].
	\]
	Recall that $\;u_1, \;u_2 \in B_{\eta},$ see (\ref{bound}).
	So, by the definition of $c^*$ and $ \mu^*$, one has
	\begin{eqnarray*}(\mathcal{E}_{a,b}-\lambda J_f-\mu \tilde J_g)(u_3) &\leq &\sup_{t\in[0,1]} (\mathcal{E}_{a,b}-\lambda J_f-\mu \tilde J_g)(\tilde \gamma(t)) \\&\leq&
		c^*+\mu^*\sup_\R h\leq c^*+1.
	\end{eqnarray*}
	Therefore,
	\[(\mathcal{E}_{a,b}-\lambda J_f)(u_3)\leq c^*+1+\mu^*\sup_\R h\leq c^*+2,\]
	and from (\ref{inclusion}) one has \[u_3\in B_r.\]
	Accordingly, we conclude that $\tilde J_g'(u_i)=J_g'(u_i)$, $i=1,2,3,$ so that $u_1,u_2, u_3$ are critical points of $\mathcal{E}_{a,b}-\lambda J_f-\mu  J_g$, i.e., weak solutions to problem $(\mathcal{P}_{a,b}^{f,g})$.	
\end{proof}

\medskip

\begin{remark}\rm 
We conclude the paper by giving an upper  estimate of $\lambda^*$ (see (\ref{lambdastar})) when $$f(x,t)=\alpha(x)h(t),$$
where $\alpha\in L^\infty(\Omega)$ and $h:\mathbb R\to \mathbb R$ is a continuous function  with $H(t_0)>0$ for some $t_0>0$, $\ds \lim_{t\to 0 }\frac{H(t)}{t^2}=0 $ and   $\ds{\rm essinf}_{x\in \Omega}\alpha=:\alpha_0>0$; hereafter, $\ds H(t)=\int_0^t h(s)ds.$ Assumption $H_1)$ is trivially verified. In order to verify $H_2)$, we consider the function 
$$u_\sigma(x)=
\left\{ 
\begin{array}{lll}
\displaystyle 0 & \mbox{if} & x\in \Omega\setminus B(x_0,R);\\
t_0\frac{(R-|x-x_0|)}{R(1-\sigma)} & \mbox{if} & x\in B(x_0,R)\setminus B(x_0,\sigma R);\\
t_0 & \mbox{if} & x\in B(x_0,\sigma R),
\end{array}\right.$$ 
where $x_0\in \Omega$, $\sigma\in (0,1)$, and $R>0$ is chosen in such a way that $R<{\rm dist}(x_0,\partial \Omega).$
It is clear that 
$$\|u_\sigma\|^2=t_0^2(1-\sigma)^{-2}(1-\sigma^d)R^{d-2}\omega_d;$$
$$\ds\int_\Omega u_\sigma^{2^*}\geq t_0^{2^*}\sigma^dR^{d}\omega_d;$$
$$\ds\int_\Omega H(u_\sigma)\geq \left[H(t_0)\sigma^d-\max_{|t|\leq t_0}H(t)(1-\sigma^d)\right]R^{d}\omega_d.$$
If $\sigma\in (0,1)$ is  close enough to $1,$ 
the right-hand side of the last estimate becomes strictly
positive; let $\sigma_0\in (0,1)$ such a value. In particular, one has that
 $$\ds \int_\Omega F(x,u_{\sigma_0})\geq \alpha_0\left[H(t_0)\sigma_0^d-\max_{|t|\leq t_0}H(t)(1-\sigma_0^d)\right]R^{d}\omega_d>0,$$
 which proves the validity of $H_2)$. Moreover, by the above estimates, it turns out that
 $$\lambda^*\leq \frac{at_0^2(1-\sigma_0)^{-2}(1-\sigma_0^d)/2+b(t_0^2(1-\sigma_0)^{-2}(1-\sigma_0^d))^2R^{d-2}\omega_d/4-t_0^{2^*}\sigma_0^dR^{2}/2^*}{\alpha_0\left[H(t_0)\sigma_0^d-\ds\max_{|t|\leq t_0}H(t)(1-\sigma_0^d)\right]R^{2}}=:\tilde \lambda.$$
 Therefore, instead of $\lambda^*$ in Theorem \ref{utolso-tetel}, we  can use the more explicit value of $\tilde \lambda>0;$ the same holds for Theorem \ref{japan} with the choice $\alpha\equiv \alpha_0=1$, $h(t)=|t|^{q-2}t$, $t_0=1$ and $\sigma_0=(3/4)^{1/d}.$
\end{remark}
 
\begin{remark}\rm 
	We conclude the paper by stating that the regularity results from Theorem \ref{main1} and the applications in Theorems  \ref{Poisson}\&\ref{japan} can be extended to compact Riemannian manifolds with suitable modifications. The most sensitive part of the proof is the equivalence from Theorem \ref{main1}/(i), which explores the \textit{non-existence} of extremal functions in the critical Sobolev embedding; such a situation is precisely described in 
	the paper of Hebey and Vaugon \cite{HV}.  
The non-compact case requires a careful analysis via appropriate group-theoretical arguments as in  Farkas and Krist\'aly \cite{FK}. We leave the details for interested readers. 
\end{remark}


\begin{thebibliography}{}

	\bibitem{ACF}
	{\sc C.O. Alves, F.J. Corr\^{e}a, G.M. Figueiredo}, {\it On a class of nonlocal elliptic problems with critical growth}. Differ. Equ. Appl. {\bf 2} (2010) 409--417.
	\bibitem{ACM}
	{\sc C.O. Alves, F.J. Corr\^{e}a, T.F. Ma}, {\it
		Positive solutions for a quasilinear elliptic equation of Kirchhoff type.}
	Comput. Math. Appl. {\bf 49} (2005) 85--93.
	
	
\bibitem{AFP} {\sc G. Autuori, A. Fiscella, P. Pucci,} \textit{Stationary Kirchhoff problems involving a fractional elliptic operator and a critical nonlinearity.} Nonlinear Anal. \textbf{125} (2015), 699--714.	

	
	\bibitem{CKW}
	{\sc C.Y. Chen, Y.C. Kuo, T.F. Wu}, {\it The Nehari manifold for a Kirchhoff type problem involving sign-changing weight functions}.
	J. Differential Equations {\bf 250} (2011) 1876--1908.
		\bibitem{CL} 
	{\sc M. Chipot, B. Lovat}, \textit{Some remarks on nonlocal elliptic and parabolic problems}. Nonlinear Anal. \textbf{30} (7) (1997) 4619--4627.
	\bibitem{CF}
	{\sc F.J. Corr\^{e}a, G.M. Figueiredo}, {\it On an elliptic equation of $p$-Kirchhoff type via variational methods}. Bull. Austral. Math. Soc. {\bf 74} (2006) 263--277.
	
%	\bibitem{Croke} {\sc C. Croke}, \textit{A sharp four-dimensional isoperimetric inequality.}
%	Comment. Math. Helv. \textbf{59} (1984), no. 2, 187--192.
	
	\bibitem{Dacorogna}
	{\sc B. Dacorogna,} {\it Direct methods in the calculus of variations.}  Second edition. Applied Mathematical Sciences, 78. Springer, New York, 2008. xii+619 pp. 
	
\bibitem{DS}	{\sc Y. Deng, W. Shuai,} \textit{Sign-changing multi-bump solutions for Kirchhoff-type equations in $\mathbb R^3.$} Discrete Contin. Dyn. Syst. \textbf{38} (2018), no. 6, 3139--3168.
	
	
	\bibitem{Fan}
	{\sc H. Fan}, {\it  Multiple positive solutions for a class of Kirchhoff type problems involving critical Sobolev exponents}. J. Math. Anal. Appl. \textbf{431} (2015) 150--168.
	
	\bibitem{FK} {\sc C. Farkas, A. Krist\'aly}, \textit{Schr\"odinger-Maxwell systems on non-compact Riemannian manifolds.} Nonlinear Anal. Real World Appl. \textbf{31} (2016), 473--491.
	
%	\bibitem{FKSZ} {\sc C. Farkas, A. Krist\'aly and A. Szak\'al}, {\it Sobolev interpolation inequalities on Hadamard manifolds,} 2016 IEEE 11th International Symposium on Applied Computational Intelligence and Informatics (SACI), Timisoara, 2016, pp. 161-166.
%	
%	
	\bibitem{F}
	{\sc G.M. Figueiredo,} {\it Existence of a positive solution for a Kirchhoff problem type with critical growth via truncation argument}. J. Math. Anal. Appl. {\bf 401} (2013) 706--713.
	\bibitem{FS}
	{\sc G.M. Figueiredo, J.R. Santos,} {\it Multiplicity of solutions for a Kirchhoff equation with subcritical or critical growth}. Differential Integral Equations 25 (2012) 853--868.
	
	
%	\bibitem{Hebey} E. Hebey, \textit{Nonlinear analysis on manifolds: Sobolev spaces and inequalities.} Courant Lecture Notes in Mathematics, 5.
%	New York University, Courant Institute of Mathematical Sciences, New
%	York; American Mathematical Society, Providence, RI, 1999.
%	
	
	\bibitem{H1}
	{\sc E. Hebey}, {\it  Compactness and the Palais-Smale property for critical Kirchhoff equations in closed manifolds}. Pacific J. Math. {\bf 280} (2016) 41--50.
	\bibitem{H2}
	{\sc E. Hebey}, {\it  Multiplicity of solutions for critical Kirchhoff type equations}.  Comm. Partial Differential Equations {\bf 41} (2016) 913--924.
	
	
	\bibitem{HV} {\sc E. Hebey, M. Vaugon,} 
	\textit{From best constants to critical functions.} 
	Math. Z. \textbf{237} (2001), no. 4, 737--767.
	
	\bibitem{Kirchhoff}  {\sc G. Kirchhoff}, \textit{Mechanik, Teubner}, Leipzig, 1883.
	
%	\bibitem{Kleiner} B. Kleiner, \textit{An isoperimetric comparison theorem.} Invent. Math. \textbf{108} (1992), no. 1,
%	37--47.
%	
	\bibitem{LLG}
	{\sc C.Y. Lei, G.S. Liu, L.T. Guo}, {\it Multiple positive solutions for a Kirchhoff type problem with a critical
		nonlinearity}. Nonlinear Anal. Real World Appl. {\bf 31} (2016) 343--355.
	\bibitem{L}
	{\sc P.L. Lions}, {\it The concentration-compactness principle in the calculus of variations. The limit case. I}. Rev. Mat. Iberoamericana {\bf 1} (1985) 145--201.
	
%	\bibitem{MR}
%	{\sc T.F. Ma, J.E. Rivera}, {\it Positive solutions for a nonlinear nonlocal elliptic transmission problem},
%	Appl. Math. Lett. {\bf 16} (2003) 243--248.
%	\bibitem{MZ}
%	{\sc  A.M. Mao, Z.T. Zhang}, {\it Sign-changing and multiple solutions of Kirchhoff type problems without the P.S. condition}, Nonlinear Anal. {\bf 70} (2009) 1275--1287.
	\bibitem{N0}
	{\sc D. Naimen}, {\it The critical problem of Kirchhoff type elliptic equations in dimension four}. J. Differential Equations {\bf 257} (2014) 1168--1193.
%	\bibitem{N}
%	{\sc D. Naimen}, {\it Positive solutions of Kirchhoff type elliptic equations involving a critical Sobolev
%		exponent}. 
%	NoDEA Nonlinear Differential Equations Appl. {\bf 21} (2014) 885--914.
	\bibitem{PZ}
	{\sc  K. Perera, Z. Zhang}, {\it Nontrivial solutions of Kirchhoff-type problems via the Yang index}. J. Differential Equations {\bf 221} (2006) 246--255.
	
	
	\bibitem{ZP} {\sc K. Perera, Z. Zhang,} {\it Sign changing solutions of Kirchhoff type problems via invariant sets of descent flow}. J. Math. Anal. Appl. {\bf 317} (2006) 456--463.
	
	\bibitem{PS}
	{\sc P. Pucci, J. Serrin}, {\it A mountain pass theorem}. J.
	Differential Equations {\bf 60} (1985) 142--149.
	
	\bibitem{R}
	{\sc B. Ricceri}, {\it Sublevel sets and global minima of coercive functionals and local minima of their perturbation}. J. Nonlinear Convex Anal. {\bf 5} (2004) 157--168.
	
	\bibitem{Ric1}
	{\sc B. Ricceri}, {\it A further three critical points theorem}. Nonlinear Anal. {\bf 71} (2009) 4151--4157.
	
		\bibitem{Talenti}{\sc G. Talenti}, {\it Best constant in Sobolev inequality}. {Ann. Mat. Pura Appl. (4)}, {\bf 110}, {1976}, {353--372}.
	
%	\bibitem{S-Tintarev} {\sc L. Skrzypczak,  C. Tintarev,} {\it A geometric criterion for compactness of invariant
%		subspaces.} Arch. Math. {\bf 101} (2013), 259--268.
%	
	\bibitem{YM}
	{\sc X. Yao, C. Mu,} {\it Multiplicity of solutions for Kirchhoff type equations involving critical Sobolev exponents in high dimension}.
	Math. Methods Appl. Sci. {\bf 39} (2016) 3722--3734.
	
	\bibitem{Zeidler} {\sc E. Zeidler}, \textit{Nonlinear functional analysis and its
	applications.} III. Variational methods and optimization. Springer-Verlag,
	New York, 1985.
	
	
	\bibitem{Willem}{\sc M. Willem}, \emph{Minimax theorems}. Progress in Nonlinear
	Differential Equations and their Applications \textbf{24}, Birkhauser,
	Boston, 1996.




\end{thebibliography}
\end{document}